\documentclass[a4paper]{amsart}
\usepackage{a4wide}
\usepackage{amssymb,amsmath,amsthm,amsfonts}
\usepackage[colorlinks]{hyperref}

\numberwithin{equation}{section}

\newtheorem{theorem}{Theorem}[section]
\newtheorem{lemma}[theorem]{Lemma}
\newtheorem{proposition}[theorem]{Proposition}

\newtheorem{corollary}[theorem]{Corollary}

\theoremstyle{definition}
\newtheorem{definition}[theorem]{Definition}
\newtheorem{remark}[theorem]{Remark}

\newcommand{\R}{\mathbb R}

\newcommand{\C}{{\rm cap}_\alpha}

\def\C{\mathbb{C}}

\def\H{\mathcal H}
\def\eps{\varepsilon}
\def\lt{\left}
\def\rt{\right}
\def\S{\mathbb{S}}
\def\les{\lesssim}
\def\ges{\gtrsim}
\def\f{\bold{f}}
\def\Exc{\mathrm{Exc}}
\def\dist{\mathrm{dist}}

\author{Micha\"el Goldman}
\address{LJLL, Universit\'e Paris Diderot, CNRS, UMR 7598, Paris, France}
\email{goldman@math.univ-paris-diderot.fr}

\author{Matteo Novaga}
\address{Dipartimento di Matematica, Universit\`a di Pisa, 56127 Pisa, Italy} 
\email{matteo.novaga@unipi.it} 

\author{Berardo Ruffini}
\address{Dipartimento di Matematica, Universit\`a degli Studi di Bologna, 40126 Bologna, Italy}
\email{berardo.ruffini@unibo.it}

\title{Reifenberg flatness for almost minimizers of the perimeter under minimal assumptions}
\keywords{Measure theoretical perimeter, almost minimizers, Reifenberg flatness}
\subjclass[2010]{49Q05; 49Q15; 49Q20}

\begin{document}
\maketitle

\begin{abstract}
The aim of this note is to prove that almost minimizers of the perimeter are Reifenberg flat, for a very weak notion of  minimality. The main observation is that smallness of the excess at some scale implies smallness of the excess at all smaller scales.
\end{abstract}

\section{Introduction}
Motivated by applications to a variational model for charged liquid drops (see \cite{MurNovRuf,MurNovRuf2,GolNovRufprep}), the aim of this note is to study the  regularity  of almost minimizers of the (anisotropic) perimeter under minimal assumptions. For a uniformly elliptic and regular anisotropy $\Phi$ (see Definition \ref{def:anis}) and a set of finite perimeter $E$, we define $P_\Phi(E)$ as 
\[
 P_\Phi(E)=\int_{\partial E} \Phi(x,\nu_E) d\H^{n-1}.
\]
Here $\nu_E$ denotes the (measure-theoretic) outward normal to $E$. For  $\Lambda,r_0>0$  we say that $E$ is a $(\Lambda,r_0)-$minimizer of the perimeter $P_\Phi$ if for every $x\in \R^n$ and $r\le r_0$,
\begin{equation}\label{def:Qmin}
 P_\Phi(E)\le P_\Phi(F)+ \Lambda r^{n-1} \qquad \text{ if }\quad E\Delta F\subset B_r(x).
\end{equation}
Our main result is the following (all constants depend implicitly on the dimension and on the given anisotropy $\Phi$)
\begin{theorem}\label{theo:main}
 For every $\delta>0$, there exists $\eps(\delta)>0$ such that if $\Lambda\le \eps$  then every $(\Lambda,r_0)-$minimizer is $\delta-$Reifenberg flat (see Definition \ref{Reifenberg flatness}) outside a singular set of vanishing $\H^{n-3}$ measure.
\end{theorem}
\begin{remark}
 In the isotropic case, the Hausdorff dimension of the singular set is smaller than $8$. 
\end{remark}

In \cite{GoMaTa} it is shown that in fact we cannot expect much more regularity under the weak minimality condition \eqref{def:Qmin}.
The classical theory for almost minimizers of the perimeter asserts that, under the stronger condition
\begin{equation}\label{eq:Qminvstrong}
 P_\Phi(E)\le P_\Phi(F)+ \Lambda r^{n-1+\alpha} \qquad \forall E\Delta F\subset B_r(x),
\end{equation}
with $\alpha\in (0,1)$, almost minimizers are $C^{1,\frac{\alpha}{2}}$ outside a small singular set (see \cite{SSA,Bombieri,Tamanini,Maggi, DPM}). The closest result to ours is the one of Ambrosio and Paolini \cite{AmPa} where they prove that  for the isotropic perimeter $P$, if $\omega(r)$ is a rate function such that $\omega(r)\to0$ as $r\to 0$ and if $E$ satisfies
\begin{equation}\label{eq:QminAmPa}
 P(E,B_r(x))\le (1+ \omega(r))P(F,B_r(x)) \qquad \forall E\Delta F\subset B_r(x),
\end{equation}
then outside a small singular set, $\partial E$ is vanishing Reifenberg flat. Our result is thus an improvement over \cite{AmPa} with a rate function which is small but not infinitesimal. Indeed, thanks to the density estimates (see Proposition \ref{density}), condition \eqref{def:Qmin} is equivalent (up to changing the value of $\Lambda$) to 
\begin{equation}\label{eq:Qministrong}
 P(E,B_r(x))\le (1+ \Lambda)P(F,B_r(x)) \qquad \forall E\Delta F\subset B_r(x).
\end{equation}

Our proof of Theorem \ref{theo:main} follows the classical route of regularity theory for minimizers of the perimeter as pioneered by De Giorgi. The main ingredient is an $\eps-$regularity result. To state it we introduce the spherical excess:
\begin{equation}\label{def:excess}
 \Exc(E,x,r)=\min_{\nu\in \S^{n-1}} \frac{1}{r^{n-1}}\int_{\partial E\cap B_r(x)}\frac{|\nu-\nu_E|^2}{2} d\H^{n-1}.
\end{equation}
When $x=0$ we simply write $\Exc(E,r)=\Exc(E,0,r)$.
\begin{theorem}\label{theo:eps}
 For every $\delta>0$ there exists $\eps(\delta)>0$ such that if $E$ is a $(\Lambda,r_0)-$minimizer with   $0\in \partial E$ and (see Definition \ref{def:anis} for the definition of $\ell$)
 \[
  \Exc(E,r)+ \Lambda + \ell r\le \eps 
 \]
 then $E$ is $(\delta, \frac{r}{2})-$Reifenberg flat in $B_{\frac{r}{2}}$.
\end{theorem}
\subsubsection*{Comments on the proof}
The proof follows the standard Campanato iteration scheme with at its heart the excess improvement by tilting lemma (see Lemma \ref{tilt} below). 
In particular, we largely stick to the arguments in \cite{Maggi} (and \cite{DPM} for the treatment of the anisotropic case) where the proof is written  in the language of sets of finite perimeters.

\begin{remark}
A technical difference between our definition of almost minimality and those in \cite{Maggi,DPM} is that our error term depends on the diameter of the perturbation, instead of the measure of the symmetric difference between the two competing sets. This induces some differences in the proof of   the density estimates in Proposition \ref{density} as well as in the outcome of the  almost harmonicity property \eqref{lip3}. 
\end{remark}

As also nicely explained in \cite[Section 7]{DHV}, the result is derived by the following main steps. The starting point is to obtain density upper and lower bounds
(see Proposition \ref{density}). The ingredients for the tilting lemma are then the following. In the small excess regime, 
we show that locally $\partial E$ can be mostly covered by the graph of a Lipschitz function $u$ with  Dirichlet energy bounded by the excess (this is done combining the height bound,  Proposition \ref{hb} with  a Lipschitz approximation argument, Proposition \ref{lipapprox}).  
Moreover the approximating function $u$ is almost harmonic (see \eqref{lip3}). Using this information and the regularity theory for harmonic functions we obtain smallness of the flatness $\f_{2,\nu}$ (see  \eqref{def:flat} for its definition) 
at a smaller scale.  The tilt lemma is then concluded by the Cacciopoli inequality, Proposition \ref{caccioppoli}, which bounds the excess by the flatness.\\
Our main observation is that while in the classical theory the tilt lemma leads to excess decay, in the sense that under condition \eqref{eq:Qminvstrong} or even \eqref{eq:QminAmPa}, the excess converges to $0$ as $r\to 0$, in our setting the excess is only going to remain small (essentially bounded by $\Lambda$). For a similar application of this idea in the context of optimal transport see \cite{goldman}.\\
Once we know that the excess and thus also the flatness remain small at all scales, the Reifenberg regularity follows from the height bound, Proposition \ref{hb} (see also \cite{AmPa}).\\

In order to conclude the proof of Theorem \ref{theo:main}, we also need to estimate the size of the singular set. This is obtained by a classical stability argument, 
essentially the continuity of the excess, which allows to transfer the known results for minimizers of the perimeter
(i.e. the case $\Lambda=0$) to the case $\Lambda>0$ small.

We close this introduction with another  simple but useful consequence of this stability argument:

\begin{corollary}\label{cor:reg}
For every $\delta>0$ there exists $\bar \Lambda=\bar \Lambda(\delta)$ such that  for every $C^1-$regular,  compact set $E$, if  $E_k$ are $(\Lambda_k,r_0)-$minimizers with $\limsup_{k\to \infty}\Lambda_k\le \bar \Lambda$ and $E_k\to E$ in $L^1$, then there exists $r>0$ such that if $k$ is large enough, $\partial E_k$ are uniformly $(\delta,r)-$Reifenberg flat.
\end{corollary}

\section{Definitions and notation}
We refer to \cite{AFP,Maggi} for the definitions and main properties of sets of finite perimeter. In particular we denote by $\partial E$ the measure-theoretic boundary and $\partial^* E$ the reduced boundary. We denote by $\nu_E$ the outward normal to $E$ and by $P(E)$ its (measure theoretical) perimeter. 

\begin{definition}\label{def:anis} We say that  $\Phi:\R^n\times\R^n\to\R^+$  is an elliptic integrand if it is   lower-semicontinous and if  $\Phi(x,\cdot)$ is a positive 1-homogeneous and convex function, that is $\Phi(x,t \nu)=t\Phi(x,\nu)$ for $t>0$. 
For  a set of finite perimeter $E$ and  an open set $U$ we define
\[
P_{\Phi}(E,U)=\int_{\partial E\cap U}\Phi(x,\nu_{E}(x))\,d\mathcal H^{{n-1}}(x).
\]
If $U=\R^n$ we simply write $P_\Phi(E)=P_\Phi(E,\R^n)$. An elliptic integrand is called a uniformly elliptic and regular anisotropy if for every $x\in \R^n$, $\Phi(x,\cdot)\in C^{2,1}(\S^{n-1})$ and there exists $\lambda\ge 1$ and $\ell\ge 0$ such that for any $x,y\in\R^n$, and $\nu,\nu'\in\mathbb S^{n-1}$ and $e\in \R^n$,
\begin{equation}\label{ellipticity}
\begin{aligned}
&\frac1\lambda\le\Phi\le\lambda\\
&|\Phi(x,\nu)-\Phi(y,\nu)|+|\nabla\Phi(x,\nu)-\nabla \Phi(y,\nu)|\le \ell |x-y| \\
& |\nabla \Phi(x,\nu)|+|\nabla^2\Phi(x,\nu)|+\frac{|\nabla^2\Phi(x,\nu)-\nabla^2\Phi(x,\nu')|}{|\nu-\nu'|}\le \lambda \\
& \nabla^2\Phi(x,\nu)e\cdot e\ge\frac{|e-(e\cdot\nu)\nu|^2}{\lambda},
\end{aligned}
\end{equation}
where $\nabla \Phi$ and $\nabla^2\Phi$ denote the gradient and the Hessian with respect to the $\nu$ variable.
\end{definition}
We will use the notation $A\les B$ to indicate that there exists a constant $C$ depending on the dimension $n$ and on the fixed anisotropy $\Phi$ through $\lambda$ such that $A\le C B$. For scaling purposes we keep the dependence in $\ell$ explicit. In some statements we will also use the notation $A\ll B$ to indicate that there exists a (typically small) universal constant $\eps>0$ such that if $A\le \eps B$ then the conclusion of the statement holds.
\begin{definition}\label{Reifenberg flatness}
 Let $\delta,r>0$ and $x\in \R^n$. We say that $E$ is $(\delta,r)-$Reifenberg flat in $B_r(x)$ if for every $B_{r'}(y)\subset B_r(x)$, there exists an hyperplane $H_{y,r'}$ containing $y$ and such that 
 \begin{itemize}
  \item we have 
   \[
  \frac{1}{r'} \dist(\partial E\cap B_{r'}(y), H_{y,r'}\cap B_{r'}(y)) \le \delta,
 \]
 where $\dist$ denotes the Hausdorff distance;
 \item one of the connected components of 
 \[
 \{ \dist(\cdot, H_{y,r'})\ge  2 \delta r'\}\cap B_{r'}(y)
 \]
is included in $E$ and the other in $E^c$.
 \end{itemize}
We say that $E$ is uniformly $(\delta,r)-$Reifenberg flat if the above conditions hold for every $x\in \partial E$.

\end{definition}
We now introduce several notions of excess and flatness that we will use. To do that we first need to fix some geometric notation. For a given point $x$, radius $r$ and direction $\nu\in \S^{n-1}$, we define the cylinder 
\[
 \C_\nu(x,r)=x+ \{y\in\R^n \ : \ |y\cdot \nu|<r, \, |y-(y\cdot \nu)\nu|<r\}.
\]
In the normalized situation where $x=0$ and $\nu=e_n$ we simply set $\C(r)=\C_{e_n}(0,r)$ and $\C=\C(1)$. We often write a point $x\in \R^n$  as $x=(x',x_n)\in\R^{n-1}\times\R$ and define $\nabla'$ the gradient with respect to the first $(n-1)-$variables. We denote by $B'_r$ the disk of radius $r$ in the plane $x_n=0$ so that $\C=B'_1\times(-1,1)$. Given $x\in \R^n$, $r>0$ and $\nu\in \S^{n-1}$, we define the cylindrical excess in the direction $\nu$ as 
\begin{equation}\label{def:cylexcess}
 \Exc_\nu(E,x,r)= \frac{1}{r^{n-1}}\int_{\partial E\cap \C_\nu(x,r)}\frac{|\nu-\nu_E|^2}{2} d\H^{n-1}.
\end{equation}
As above, if $x=0$ and $\nu=e_n$ we simply write $\Exc_n(E,r)=\Exc_{e_n}(E,0,r)$. Notice that $\Exc(E,x,r)\le \Exc_\nu(E,x,r)$. We define the cylindrical ($L^2$) flatness as 
\begin{equation}\label{def:flat}
 \f_{2,\nu}(E,x,r)=\frac{1}{r^{n-1}}\min_{h\in \R}\int_{\partial E\cap \C_\nu(x,r)}\frac{|(y-x)\cdot\nu -h|^2}{r^2}\,d\mathcal H^{n-1}(y).
\end{equation}

\section{Density  estimates}
In this section we show  density estimates for $(\Lambda,r_0)-$minimizers. Let us first point out that under the minimality condition \eqref{eq:Qministrong} these have already been established, see \cite{Maggi}. It is well-known that minimality conditions such as \eqref{eq:Qminvstrong} need to be treated with extra-care since the error term does not go to zero for fixed $r$ as $|E\Delta F|\to 0$. 
The standard proof of density estimates for almost minimizers in the sense \eqref{eq:Qminvstrong} goes through a monotonicity formula (see \cite[Lemma 2]{Tamanini}). Under the weaker condition \eqref{def:Qmin}, this monotonicity formula would lead to a logarithmically failing estimate and we need to argue differently.

\begin{proposition}\label{density}
There exists a universal $\eps>0$ such that if $\Lambda<\eps$ and  $E$ is a $(\Lambda,r_0)-$minimizer  then  for every $x\in \partial E$ and every $0<r\le r_0$
\begin{equation}\label{dens1}
\min(|E\cap B_r(x)|,|B_r(x)\backslash E|)\ges r^n 
\end{equation}
and 
\begin{equation}\label{dens2}
r^{n-1}\les P(E,B_r)\les r^{n-1}.
\end{equation}
\end{proposition}

\begin{proof}
We  argue as in the proof of Proposition 6.6 in \cite{DHV}.
By translation and density we may assume without loss of generality that $x=0$ and $0\in \partial^*E$. We first prove \eqref{dens2}. The upper bound is easy. Using the almost minimality property \eqref{def:Qmin} with $E\backslash B_r$ we obtain 
\[
 P(E,B_r)\les P_\Phi(E,B_r)\les \H^{n-1}(\partial B_r\cap E) +\Lambda r^{n-1}\les r^{n-1}.
\]
 We thus turn our attention to the lower bound. Let $\theta\in(0,1/2)$ be fixed  and then $\eta=\eta(\theta) $ to be chosen below. We claim that if 
\begin{equation}\label{hyp:contrad}
 \frac{1}{r^{n-1}}P(E,B_r)\le \eta
\end{equation}
then there exists $C>0$ such that 
\begin{equation}\label{toproveiterdens}
 \frac{1}{(\theta r)^{n-1}}P(E,B_{\theta r})\le \theta \frac{1}{r^{n-1}}P(E, B_r) + C\Lambda.
\end{equation}
Indeed, if \eqref{hyp:contrad} holds then by the relative isoperimetric inequality,
\[
 \min\lt( \frac{|E\cap B_r|}{r^n}, \frac{| B_r\backslash E |}{r^n}\rt)\les \lt(\frac{1}{r^{n-1}}P(E,B_r)\rt)^{\frac{n}{n-1}}\les \eta^{\frac{1}{n-1}} \frac{1}{r^{n-1}}P(E,B_r).
\]
Assume first that $|E\cap B_r|\le |B_r\backslash E|$. Choose then $t\in (\theta r, 2\theta r)$ such that 
\begin{align*}
 \H^{n-1}(E\cap \partial B_t)&\le \frac{1}{\theta r}\int_{\theta r}^{2\theta r} \H^{n-1}(E\cap \partial B_s) ds \\
 & \les \frac{|E\cap B_r|}{\theta r} \les \theta^{-1} \eta^{\frac{1}{n-1}}  P(E,B_r).
\end{align*}
Testing the almost minimality property \eqref{def:Qmin} with $F= E\backslash B_t$ we then get 
\[
 P(E,B_t)\les P_\Phi(E,B_t)\les \H^{n-1}(E\cap \partial B_t)+ \Lambda t^{n-1}\les \theta^{-1} \eta^{\frac{1}{n-1}}  P(E,B_r) + \Lambda (\theta r)^{n-1}.
\]
Since $P(E,B_{\theta r})\le P(E,B_t)$, we get
\begin{equation}\label{intermediate}
 \frac{1}{(\theta r)^{n-1}}P(E,B_{\theta r})\le  C\lt(\theta^{-n}\eta^{\frac{1}{n-1}} \frac{1}{r^{n-1}} P(E,B_r) + \Lambda\rt).
\end{equation}
If instead $|E\cap B_r|\ge |B_r\backslash E|$, we argue exactly in the same way using $E^c$ instead of $E$ (and noting that if $E$ satisfies \eqref{def:Qmin} then also $E^c$ satisfies it). Choosing now 
\[\eta=\min\lt(\frac{\omega_{n-1}}{2}, \lt(\frac{1}{C} \theta^{n+1}\rt)^{n-1}\rt),\]
so that in particular  $C\theta^{-n}\eta^{\frac{1}{n-1}}\le\theta$, we conclude the proof of \eqref{toproveiterdens}.
Assume now that 
\[
 \Lambda \le \frac{\eta(1-\theta)}{C}
\]
where $C$ is the constant appearing in \eqref{toproveiterdens}. Then if \eqref{hyp:contrad} holds, by \eqref{toproveiterdens}
\[
 \frac{1}{(\theta r)^{n-1}}P(E,B_{\theta r})\le \theta \eta + C\Lambda\le \eta,
\]
and we can iterate to obtain 
\[
 \limsup_{k\to \infty} \frac{1}{(\theta^k r)^{n-1}}P(E,B_{\theta^k r})\le\eta.
\]
Since $\eta< \omega_{n-1}$,  this contradicts $\lim_{r\to 0} \frac{1}{r^{n-1}}P(E,B_r)=\omega_{n-1}$ (which follows from the hypothesis $0\in \partial^* E$)  so that the lower bound in \eqref{dens2} is proven.\\

We now prove \eqref{dens1}. For this we will show that there is $\eta'$ such that if 
\begin{equation}\label{hypdens}
 \min(|E \cap B_r|, |B_r\backslash E|)\le \eta' r^n
\end{equation}
then we reach a contradiction with \eqref{dens2}. Indeed, assume that \eqref{hypdens} holds. Arguing as above we find $t\in (\theta r, 2\theta r)$ such that 
\[
 P(E,B_t)\les \frac{\eta'}{\theta} r^{n-1}\le C_\theta \eta' t^{n-1}.
\]
Taking $\eta'$ small enough we reach the desired contradiction.

\end{proof}

\section{Excess decay}

The aim of this section is to prove the following result:
\begin{proposition}\label{excessdecay}
There exists $\eps_{\rm dec}>0$ such that if $\eps\in(0,\eps_{\rm dec})$ and  $E$ is a $(\Lambda,r_0)-$minimizer with $0\in \partial E$ and 
\[
 \Exc(E,r)+ \ell r+ \Lambda\le \eps
\]
then for $r'\in(0,r)$,
\begin{equation}\label{conclusiondecay}
 \Exc(E,r')\les \eps.
\end{equation}

\end{proposition}

The   strategy follows the steps described in the introduction. Since most of the proof is identical to the existing literature (see in particular \cite{Maggi}) we only point out the main differences.

\subsection{Height bound}
We start with the height bound. Without loss of generality, we may assume that $x=0$ and $\nu=e_n$.

\begin{proposition}\label{hb}[Height Bound]
For every  $\delta\in(0,1/4)$ there exists $\eps_{\rm hb}=\eps_{\rm hb}(\delta)>0$ such that if  $E$ is a $(\Lambda,r_0)-$minimizer and 
\[
 \Exc_n(E,2r)+ \ell r<\eps_{\rm hb},
\]
 then 
\[
\sup \left\{ |x_n|\,:\, x\in \C(r)\cap \partial E \right\}\le\delta r,
\]
and
\[
\left| \left\{x\in \C(r)\cap  E\,:\, x_n>\delta r  \right\} \right|+\left| \left\{x\in \C(r)\cap  E^c\,:\, x_n<-\delta r  \right\} \right|=0.
\]
\end{proposition}

\begin{proof}
The proof is exactly as in \cite[Lemma 4.1]{DPM}. The only difference is for the   compactness   properties of almost minimizers. In our case this may be obtained as in \cite[Proposition 1.4]{AmPa}. See also \cite[Lemma 7.2]{DHV}, where it is observed that the height bound is a direct consequence of the density estimates.
\end{proof}

\begin{remark}\label{rem:Reifenberg}
 We stress that Proposition \ref{hb} in particular shows that if the excess is small at every scale around a given point, then $E$ is Reifenberg flat at that point.
\end{remark}

\subsection{Lipschitz approximation}
The second ingredient  is the fact that as long as the excess is small  inside a cylinder, the boundary of $E$ can be approximated by a Lipschitz function.
\begin{proposition}\label{lipapprox}[Lipschitz approximation]
For every  $\delta\in(0,1/4)$ there exists $\eps_{\rm lip}=\eps_{\rm lip}(\delta)>0$ and  $\sigma>0$ such that if $E$ is a $(\Lambda,r_0)-$minimizer with 
\[
 \Exc_n(E,2r)+ \ell r\le \eps_{\rm lip},
\]
then there exists $u:\R^{n-1}\to\R$ such that, if 
\[
M=\C(r)\cap \partial E,
\]
\[
M_0=\left\{y\in M\,:\, \sup_{s\in(0,r)}\Exc_n(E,y,s)\le\sigma \right\}
\]
and $\Gamma$ is the graph of $u$ in $B'_r$,  then 

\begin{equation}\label{lip1}
  \sup_{B'_r}|u|\le \delta r,\qquad Lip(u)\le1,\qquad M_0\subset M\cap \Gamma, \qquad  \frac{\mathcal H^{n-1}(M\Delta\Gamma)}{r^{n-1}}\les \Exc_n(E,2r),
\end{equation}
and
\begin{equation}\label{lip2}
\frac{1}{r^{n-1}} \int_{B'_r}|\nabla' u|^2\les \Exc_n(E,2r).
\end{equation}
Moreover, for all $\varphi\in C^1_c(B'_r)$

\begin{equation}\label{lip3}
\frac{1}{r^{n-1}} \int_{B'_r} \nabla^2\Phi(0,e_n)(\nabla' u,0)\cdot(\nabla' \varphi,0)\les \|\nabla'\varphi \|_\infty (\Exc_n(E,2r)+{\sqrt\Lambda }+ \ell r).
\end{equation}

\end{proposition}
\begin{proof}
Up to choosing $\eps_{\rm lip}\le\eps_{\rm hb}$, we can repeat \emph{Step one} of Lemma 4.3 in \cite{DPM} to show the existence of a Lipschitz function  $u$ satisfying properties \eqref{lip1} and \eqref{lip2}. Letting $\Phi_0=\Phi(0,\cdot)$, to prove \eqref{lip3} we show that
\begin{equation}\label{toprove}
\frac{1}{r^{n-1}} \int_{\C(r)\cap\partial E} \nabla\Phi_0(\nu_E)\cdot (\nabla'\varphi ,0) (\nu_E\cdot e_n)\,d\mathcal H^{n-1}   \les \|\nabla' \varphi\|_\infty(\sqrt\Lambda+\ell r).
\end{equation}

\noindent
With this inequality at hand we may  conclude by following verbatim the proof of \cite[Lemma 4.3]{DPM}.

To get \eqref{toprove} we notice that by scaling we may assume that $r=\|\nabla'\varphi\|_\infty=1$. Let  $t>0$ and set $f_t(x)=x+t\alpha( x_n)\varphi( x') e_n$ where $\alpha\in C^1_c([-1,1],[0,1])$ with $\alpha=1$ on $[-1/2,1/2]$ and $|\alpha'|<3$.  Then, up to choosing $t$ small enough, we get that $f_t$ is a diffeomorphism of $\R^n$ and $f_t(E)\Delta E\subset \C$. By the $(\Lambda,r_0)-$minimality \eqref{def:Qmin} of $E$, we get that
\[
P_\Phi(E)\le P_\Phi(f_t(E))+\Lambda.
\]  
Using a  Taylor expansion and the area formula as  in \cite[page 526]{DPM}, we find
\[
P_\Phi(f_t(E))-P_\Phi(E)\le -t\int_{\C\cap\partial E} \nabla\Phi_0(\nu_E)(\nabla'\varphi,0)(\nu_E\cdot e_n)\,d\mathcal H^{n-1}+ C(\ell |t|+t^2).
\]
Combining the last two inequalities we get  
\[
\int_{\C\cap\partial E} \nabla\Phi_0(\nu_E)\cdot (\nabla'\varphi ,0) (\nu_E\cdot e_n)\,d\mathcal H^{n-1} \les\frac{\Lambda}{|t|}+|t|+\ell. 
\]
The proof of  \eqref{toprove} is concluded  by choosing $t=\sqrt\Lambda$.

\end{proof}

\subsection{Caccioppoli inequality}

We now turn to the Cacciopoli inequality which allows to control the excess by the flatness.
\begin{proposition}\label{caccioppoli}[Caccioppoli inequality]
There exists $\eps_{ca}>0$  such that if 
\[
 \Exc_\nu(E,4r)+ \Lambda+\ell r\le \eps_{\rm ca}
\]
 then (recall the definition \eqref{def:flat} of the flatness)
\begin{equation}\label{e<f}
\Exc_\nu(E,r)\les \f_{2,\nu}(E,2 r)+\Lambda+\ell r.
\end{equation}
\end{proposition}

\begin{proof}
By scaling it is enough to prove the result for $r=1$. Checking the proof of \cite[Lemma 4.4]{DPM} (see also \cite[Lemma 24.9]{Maggi} in the isotropic case) it can be seen  that almost minimality is only used there to derive  \cite[(4.70)]{DPM}. Since this would also be a consequence of the almost minimality property \eqref{def:Qmin}, the proof follows by repeating verbatim the proof of  \cite[Lemma 4.4]{DPM}.
\end{proof}

\subsection{Tilt Lemma}
We now combine the previous results to prove the anticipated  \emph{tilt} lemma, stating that if the cylindrical excess is small at a given scale then, up to a rotation it  remains small at smaller scales.
\begin{lemma}\label{tilt}[Tilt Lemma]
For every $\theta\in(0,1)$ small enough, there exists $\eps_{\rm tilt}=\eps_{\rm tilt}(\theta)>0$ and $C_\theta>0$ such that if $E$ is a $(\Lambda,r_0)-$minimizer and 
\[
 \Exc_n(E,r)+ \Lambda+\ell r\le \eps_{\rm tilt},
\]
then there exists $\nu\in \S^{n-1}$ such that 
\begin{equation}\label{eq:tilt}
 \Exc_{\nu}(E, \theta r)\les \theta^2 \Exc_n(E,r) + C_\theta \Lambda + \ell \theta r.
\end{equation}
%
%
\end{lemma}

\begin{proof}
We follow relatively closely the proof in \cite[Section 25.2]{Maggi} (see also \cite[Lemma 4.6]{DPM}). 
By scaling, we may assume that $r=1$. 
Up to choosing $\eps_{\rm tilt}$ smaller than $\eps_{\rm lip}$ and $\eps_{\rm hb}$, we may apply Lemmas \ref{hb} and \ref{lipapprox}. Define the $(n-1)\times(n-1)$ matrix $A$ with components  $A_{i,j}=\nabla^2\Phi(0,e_n) e_i\cdot e_j$, for $i,j=1,\ldots, n-1$. 
We recall the following simple result from elliptic regularity theory (see \cite[Lemma 4.7]{DPM}). For every $\tau>0$, there exists $\eps_{\rm har}=\eps_{\rm har}(\tau)$ such that 
if $u$ is such that for every $\varphi\in C^1_c(B'_1)$,
\[
  \int_{B'_1} |\nabla'u|^2\le 1 \qquad \textrm{and } \qquad \int_{B'_1} (A\nabla'u)\cdot \nabla'\varphi\les \eps_{\rm har} \|\nabla'\varphi\|_\infty,
\]
 then there exists $v$ such that  for every $\varphi\in C^1_c(B'_1)$
\[
 \int_{B'_1}|u-v|^2\le\tau, \quad \int_{B'_1}|\nabla'v|^2\le 1 \quad \textrm{and } \quad \int_{B'_1} (A\nabla'v)\cdot \nabla'\varphi=0.
\]
 Let now $\tau>0$ be a small parameter to be chosen depending on $\theta$ and let $\eta$ be another small parameter to be fixed depending on $\tau$ (and thus ultimately on $\theta$). Let $C>0$ be a large (universal) constant
and set 
\begin{equation}\label{def:chi}
\chi=C(\Exc_n(E,1)+\eta^{-1}\Lambda+\ell ). 
\end{equation}
  Let $u$ be the function given by Proposition \ref{lipapprox}. Setting $u_0=\chi^{-\frac{1}{2}}u$, we get from \eqref{lip2} and \eqref{lip3} that for any $\varphi\in C^1_c(D)$
\begin{multline*}
\int_{B'_1} |\nabla'u_0|^2\le 1 \qquad \textrm{ and }\\ \int_{B'_1} A\nabla'u_0\cdot\nabla'\varphi\les \|\nabla'\varphi\|_\infty\frac{\Exc_n(E,1)+ \sqrt{\Lambda} + \ell}{\sqrt{\Exc_n(E,1)+ \eta^{-1}\Lambda+\ell}}\les \|\nabla'\varphi\|_\infty(\Exc_n^{\frac{1}{2}}(E,1)+ \eta^{\frac{1}{2}}+ \ell^{\frac{1}{2}}).
\end{multline*}
Therefore, if $\eta$ is chosen small enough so that 
\[
 \Exc_n^{\frac{1}{2}}(E,1)+ \eta^{\frac{1}{2}}+ \ell^{\frac{1}{2}}\ll \eps_{\rm har},
\]
We may apply the result quoted above to find a $A-$harmonic function $v_0$ such that 
\[
 \int_{B'_1}|u_0-v_0|^2\le\tau, \quad \int_{B'_1}|\nabla'v_0|^2\le 1.
\]
Letting $v=\chi^{\frac{1}{2}} v_0$, this is equivalent to

\begin{equation}\label{v}
\int_{B'_1}|u-v|^2\le \tau\chi,\qquad\int_{B'_1} |\nabla' v|^2\le\chi.
\end{equation}
Consider $w(z)=v(0)+\nabla' v(0)\cdot z$ the tangent map at $0$ to  $v$. By standard elliptic regularity we have for any $\theta<1/2$,
\[
\sup_{B_\theta'}|v-w|^2\les \theta^4 \int_{B'_1}|\nabla'v|^2 \les \theta^4\chi.
\]
We now choose $\tau=\theta^{n+3}$ and combine the last inequality together with \eqref{v} and triangle inequality to obtain 
\begin{equation}\label{eq:centraltilt}
\frac{1}{\theta^{n+1}}\int_{B'_\theta}|u-w|^2\les \lt(\frac{\tau}{\theta^{n+1}}+\theta^2 \right)\chi\les \theta^2\chi.
\end{equation}
Finally, we set the vector $\nu$ to be the unit vector orthogonal to the graph of $v$ at the origin i.e.
\[
\nu=\frac{(-\nabla' v(0),1)}{\sqrt{1+|\nabla' v(0)|^2}}.
\]
A key observation is that with this choice of the normal,  the tilting is small, 
\begin{equation}\label{eq:smallnu}
|\nu-e_n|^2\les |\nabla'v(0)|^2\les \int_{B'_1} |\nabla' v|^2\les \chi.
\end{equation} 
Hence, since $\C_\nu(\theta)\subset \C(2\theta)$, we get
\[
\begin{aligned}
\Exc_{\nu}(E,\theta)&=\frac{1}{\theta^{n-1}}\int_{\partial E\cap \C_\nu(\theta)}\frac{|\nu_E-\nu|^2}{2} d\H^{n-1}\\
&\le \frac{1}{\theta^{n-1}}\int_{\partial E\cap \C(2\theta)}\frac{|\nu_E-\nu|^2}{2}d\H^{n-1}\\
&\les \frac{1}{\theta^{n-1}}\int_{\partial E\cap \C(2\theta)}\frac{|\nu_E-e_n|^2}{2} d\H^{n-1}+\frac{1}{\theta^{n-1}}P(E,\C(2\theta))|\nu-e_n|^2 \\ 
&\stackrel{\eqref{dens2}\&\eqref{eq:smallnu}}{\les} \Exc_n(E,2\theta)+\chi\\
&\les \theta^{-(n-1)}\Exc_n(E,1)+\chi.
\end{aligned}
\]
Possibly further reducing the value of $\eps_{\rm tilt}$ so that the right-hand side of the previous inequality lies below $\eps_{\rm ca}$, we can thus  apply Proposition \ref{caccioppoli} to $\Exc_\nu(E,\theta)$ and obtain (up to renaming $\theta$)
\[
\Exc_\nu(E,\theta)\les\f_{2,\nu}(E,2\theta)+\Lambda+ \ell \theta.
\]
The proof is then concluded arguing exactly as in the proof of \cite[ (4.122)]{DPM}, or, in the isotropic case $\ell=0$, as in \cite[ (25.22), page 343]{Maggi}, using \eqref{lip1}, \eqref{eq:centraltilt} and \eqref{eq:smallnu} 
to obtain 
\begin{equation}\label{poincarefe}
\f_{2,\nu}(E,2\theta)\les\theta^2 \chi ,
\end{equation}
which recalling the definition \eqref{def:chi} of $\chi$, immediately implies \eqref{eq:tilt}.
\end{proof}

We can finally conclude the proof of Proposition \ref{excessdecay}.
\begin{proof}[Proof of Proposition \ref{excessdecay}]
 We first translate the conclusion of Lemma \ref{tilt} from the cylindrical excess to the spherical excess. We claim that for every $\theta\in(0,1)$ small enough, there exists $\eps_{\rm it}>0$ and $C_\theta>0$ such that if 
 \begin{equation}\label{hyp:smallexcess}
  \Exc(E,r)+ \Lambda + \ell r\le \eps_{\rm it},
 \end{equation}
then 
\begin{equation}\label{excessdecayonestep}
 \Exc(E,\theta r)\le C(\theta^2 \Exc(E,r) + \ell \theta r)+ C_\theta \Lambda.
\end{equation}
Indeed, if \eqref{hyp:smallexcess} holds, up to a rotation we may assume that the infimum defining $\Exc(E,r)$ (see \eqref{def:excess}) is attained at $e_n$. Since $\C(\frac{1}{\sqrt{2}}r)\subset B_r$, this implies that $\Exc_n(E,\frac{1}{\sqrt{2}}r)\les \Exc(E,r)$  so that we may apply Lemma \ref{tilt} and obtain \eqref{excessdecayonestep} (recall that $\Exc(E,\theta r)\le \Exc_\nu(E,\theta r)$).\\
Fix now $\theta$ such that $C\theta\le \frac{1}{2}$ and let $\bar C=1+ 2 C_\theta$. Finally choose $\eps_{\rm dec}>0$ such that $(\bar C +1)\eps_{\rm dec} \le \eps_{\rm it}$. We claim that if $\eps\in(0,\eps_{\rm dec})$ and 
\[
 \Exc(E,r)+ \Lambda + \ell r\le \eps,
\]
then for every $k\ge 0$
\begin{equation}\label{eq:smallexcess}
 \Exc(E, \theta^k r)\le \bar C \eps. 
\end{equation}
Indeed, by hypothesis it holds for $k=0$. If it holds for $k$, then by the choice of parameters,
\[
 \Exc(E,\theta^k r)+ \Lambda + \ell \theta^k r\le (\bar C+1)\eps_{\rm dec}\le \eps_{\rm it}
\]
so that by \eqref{excessdecayonestep}
\[
 \Exc(E, \theta^{k+1} r)\le  C(\theta^2 \Exc(E,\theta^k r) + \ell \theta^k r)+ C_\theta \Lambda\le \frac{\bar C}{2} \eps + \frac{1}{2} \eps + C_\theta \eps
 \le \bar C \eps,
\]
proving \eqref{eq:smallexcess}. Finally if $r'\in (\theta^{k+1} r,\theta^k r)$
\[
 \Exc(E,r')\les \frac{1}{\theta^{n-1}} \Exc(E,\theta^k r),
\]
which concludes the proof of \eqref{conclusiondecay}.

\end{proof}
\section{Proof of the main results}
We begin by proving the $\eps-$regularity result in  Theorem \ref{theo:eps}.
\begin{proof}[Proof of Theorem \ref{theo:eps}]
 For fixed $\delta>0$, let $\eps\ll \min (\eps_{\rm dec}, \eps_{\rm hb}(\delta))$. Since for $x\in B_{\frac{r}{2}}$,
 \[
  \Exc\lt(E,x, \frac{r}{2}\rt)\les \Exc(E,r),
 \]
if $\Exc(E,r)$ is small enough, we can apply Proposition \ref{excessdecay} at every $x\in B_{\frac{r}{2}}$ to get that
\[
 \Exc(E,x,r')\les \eps \qquad \forall x\in B_{\frac{r}{2}}, \ \textrm{ and } r'\le \frac{r}{2}.
\]
By the height bound Proposition \ref{hb} (see also Remark \ref{rem:Reifenberg}), this proves that $E$ is $(\delta, \frac{r}{2})-$Reifenberg flat  in $B_{\frac{r}{2}}$.
\end{proof}

We now prove our main result.
 
\begin{proof}[Proof of Theorem \ref{theo:main}]
The proof follows a classical scheme, see \cite[Theorem 28.1]{Maggi} or \cite[Theorem 1.1]{DHV}. Since the statement is local, we may assume without loss of generality that $E$ is bounded. For $\eps\ll1$ we let 
 \[
  \Sigma_\eps(E)=\lt\{ x\in\partial E\,:\, \limsup_{r\to0} \Exc(E,x,r)>\eps \right\}.
 \]
By Theorem \ref{theo:eps}, we know that if $\eps$ is small enough then  $\partial E\backslash \Sigma_\eps(E)$ is $\delta-$Reifenberg flat and it is thus enough to prove that $\H^{n-3}(\Sigma_\eps(E))=0$ if $\Lambda$ is small enough. For this we argue by contradiction and assume that there exists a sequence $E_k$ of $(\Lambda_k,r_0)-$minimizers with $\Lambda_k\to 0$ and $\H^{n-3}(\Sigma_\eps(E_k))>0$. Denoting by  $\H_\infty^{n-3}$ the infinity Hausdorff pre-measure, by \cite[Lemma 28.14]{Maggi} there exists a sequence $x_k\in \partial E_k$ and radii $r_k\to 0$ such that 
\[
 \limsup_{k\to \infty} \frac{\H_\infty^{n-3}(\Sigma_\eps(E_k)\cap B_{r_k}(x_k))}{r_k^{n-3}}\ges 1.
\]
Letting $F_k= r_k^{-1}(E_k-x_k)$, we have that $F_k$ are again $(\Lambda_k,r_0)-$minimizers with
\begin{equation}\label{eq:togetabsurd}
 \limsup_{k\to \infty} \H_\infty^{n-3}(\Sigma_\eps(F_k)\cap B_{1})\ges 1.
\end{equation}
Up to  a subsequence, $x_k\to x_\infty$. Letting $\Phi_\infty(\nu)=\Phi(x_\infty, \nu)$, by the perimeter bound \eqref{dens2}, up to  a further subsequence, $F_k\to F$. By simple comparison arguments it is then not difficult to prove that  $F$ is a minimizer of $P_{\Phi_\infty}$ with $P(F_k,B_r(x))\to P(F,B_r(x))$ for every $B_r(x)$ (see \cite[Theorem 21.14]{Maggi} or \cite{GoMaTa}).  Let $\Sigma(F)=\lt\{x \in \partial F \ : \ \limsup_{r\to0} \Exc(F,x,r)>0 \right\}$ be the singular set of $F$. By the regularity theory for minimizers of uniformly elliptic and  regular anisotropic perimeter $\H^{n-3}(\Sigma(F))=0$, see \cite{SSA} (in the isotropic case this may be improved to $\H^s(\Sigma(F))=0$ for every $s>n-8$, see \cite[Theorem 28.1]{Maggi}). Thus, by definition of the Hausdorff measure, for every $\eta>0$ there exists an open set $\Omega_\eta$ such that 
\[
 \Sigma(F)\subset \Omega_\eta \qquad \textrm{and } \qquad \H^{n-3}_\infty(\Omega_\eta)\le \eta.
\]
We claim that if $k$ is large enough then $\Sigma_\eps(F_k)\subset \Omega_\eta$ which implies $\H^{n-3}_\infty(\Sigma_\eps(F_k))\le \eta$ in  contradiction with \eqref{eq:togetabsurd}. Indeed, otherwise there would exist a sequence $y_k\in \Sigma_\eps(F_k)$ with $y_k\to y$ and $y\notin \Sigma(F)$. By the perimeter convergence and the lower bound in \eqref{dens2} we find that $y\in \partial F$. Therefore, by classical $\eps-$regularity theory (see for example \cite[Theorem 3.1]{DPM}), for every $\eps'\ll1$, there exists $r>0$ such that 
\[
 \Exc(F,y,r)\le \eps'.
\]
Thanks to the perimeter convergence, we also have convergence of the excess (see for instance \cite[Proposition 22.6]{Maggi}) and thus for $k$ large enough,
\[
 \Exc(F_k,y_k,r)\le \Exc(F,y,r)+\eps'\le 2 \eps'.
\]
If we choose $\eps'\ll \eps$, by  Proposition  \ref{excessdecay} on the excess decay, we find that $y_k\notin \Sigma_\eps(F_k)$ which is a contradiction.
\end{proof}

We finally prove  Corollary \ref{cor:reg}.

\begin{proof}[Proof of Corollary \ref{cor:reg}]
 Given $\delta$, let $\eps=\eps(\delta)$ be given by Theorem \ref{theo:eps} and  set $\bar \Lambda=\frac{1}{2}\eps$. Since $E$ is compact and smooth at each point  $x\in\partial E$, $\lim_{r\to 0} \Exc(E,x,r)=0$. Therefore, for every $\eps'\ll \eps$, we can find $r>0$ and points $x_i\in \partial E$ such that for all $i$,
 \[
  \Exc(E,x_i, 4r)\le \frac{\eps'}{2}  \qquad \textrm{and } \qquad \partial E\subset \cup_i B_{\frac{r}{4}}(x_i).
 \]
If now $E_k$ is a sequence of $(\Lambda_k,r_0)-$minimizers converging in $L^1$ to $E$, by the uniform density bounds \eqref{dens1} and \eqref{dens2} of Proposition \ref{density}, we obtain the Hausdorff convergence of $\partial E_k$ to $\partial E$. In particular, for $k$ large enough we get that
\[
 \partial E_k \subset \cup_i B_{\frac{r}{2}}(x_i).
\]
Moreover, if $k$ is large enough we have $\Lambda_k\le 2 \bar \Lambda\le \eps$, 
and arguing as above using perimeter convergence we can find points $x_{i,k}\in \partial E_k\cap B_{\frac{r}{2}}(x_i)$ such that 
\[
\Exc(E_k, x_{i,k},2r)\les \Exc(E_k,x_i,4r)\le \eps'.
\]
By choosing $\eps'$ small enough, we may thus apply Theorem \ref{theo:eps} at every $x_{i,k}$ to obtain that $\partial E_k$ is $(\delta,r)-$Reifenberg flat  in $\cup_{i} B_{r}(x_{i,k})\supset \partial E$.
\end{proof}

\section*{Acknowledgments}
M.G. and B.R acknowledge partial support from the ANR-18-CE40-0013 SHAPO financed by the French Agence Nationale de la Recherche (ANR). 
M.G. was also supported by  the Investissement d'avenir project, reference ANR-11-LABX-0056-LMH, LabEx LMH.
M.N. was supported by the PRIN Project 2019/24. M.N. and B.R. are members of the INDAM-GNAMPA.

\bibliographystyle{acm}

\bibliography{biblio}
\end{document}